\documentclass[12pt, reqno]{amsart}
\usepackage{ amsmath,amsthm, amscd, amsfonts, amssymb, graphicx, color}
\usepackage[bookmarksnumbered, colorlinks, plainpages]{hyperref}
\newtheorem{thm}{Theorem}[section]
\newtheorem{cor}[thm]{Corollary}

\newtheorem{prop}[thm]{Proposition}

\newtheorem{exam}[thm]{Example}
\numberwithin{equation}{section}

\begin{document}

\title{On Hirano inverses in rings}

\author{Huanyin Chen}
\author{Marjan Sheibani}
\address{
Department of Mathematics\\ Hangzhou Normal University\\ Hang -zhou, China}
\email{<huanyinchen@aliyun.com>}
\address{
Women's University of Semnan (Farzanegan), Semnan, Iran}
\email{<sheibani@fgusem.ac.ir>}

\subjclass[2010]{15A09, 16E50, 	15A23.} \keywords{Drazin inverse, nilpotent; tripotent; Clines formula; matrix ring.}

\begin{abstract}
We introduce and study a new class of Drazin inverses. An element $a$ in a ring has Hirano inverse $b$ if
$a^2-ab\in N(R), ab=ba~\mbox{and}~b=bab.$  Every Hirao inverse of an element is its Drazin inverse.
We derive several characterizations for this generalized inverse.
An element $a\in R$ has Hirano inverse if and only if $a^2$ has strongly Drazin inverse, if and only if $a-a^3\in N(R)$.
If $\frac{1}{2}\in R$, we prove that $a\in R$ has Hirano inverse if and only if there exists $p^3=p\in comm^2(a)$ such that $a-p\in N(R)$, if and only if
there exist two idempotents $e,f\in comm^2(a)$ such that $a+e-f\in N(R)$. Clines formula and additive results for this generalized inverse are thereby obtained.
\end{abstract}

\maketitle

\section{Introduction}

Let $R$ be an associative ring with an identity. An element $a$ in $R$ is said to be Drazin invertible provided that there is a common solution to
the equations
$a^k=a^{k+1}x, x=xax, ax=xa$ for some $k\in {\Bbb N}$. As is well known, an element $a\in R$ has Drazin inverse if there exists $b\in R$ such that $$a-a^2b\in N(R), ab=ba~\mbox{and}~b=bab.$$  The preceding $b$ is unique, if such element exists. As usual,
it will be denoted by $a^D$, and called the Drazin inverse of $a$.

The Drazin inverse is an important tool in ring theory and Banach algebra. It is very useful in matrix theory and in various applications in matrices (see~\cite{W, Y, WD}). The purpose of this paper is to introduce and study a new class of Drazin inverses. An element $a\in R$ has Hirano inverse if there exists $b\in R$ such that $$a^2-ab\in N(R), ab=ba~\mbox{and}~b=bab.$$  The preceding $b$ shall be unique, if such element exists, we denote it by $a^H$, and call the
Hirano inverse of $a$. We observed these inverses form a class of Drazin inverses which is related to tripotents. Here, $p\in R$ is a tripotent if $p^3=p$ (see~\cite{AC}).

Following ~\cite{W}, An element $a\in R$ has strongly Drazin inverse $b$ if
$$a-ab\in N(R), ab=ba~\mbox{and}~b=bab.$$
In Section 2, We prove that every Hirano inverse of an element is its Drazin inverse.
An element $a\in R$ has Hirano inverse if and only if $a^2$ has strongly Drazin inverse,

In Section 3 we characterize Hirano inverses via tripotents and nilpotents. It is shown that an element $a\in R$ has Hirano inverse if and only if $a-a^3\in N(R)$.
If $\frac{1}{2}\in R$, we prove that $a\in R$ has Hirano inverse if and only if there exists $p^3=p\in comm^2(a)$ such that $a-p\in N(R)$, if and only if
there exist two idempotents $e,f\in comm^2(a)$ such that $a+e-f\in N(R)$.

Let $R$ be a ring, and let $a,b,c\in R$ with $aba=aca$. In Section 4, we establish Clines formula for Hirano inverses. We prove that
$ac$ has Hirano inverse if and only if $ba$ has Hirano inverse. In the last section, we are concern on generalized Jacobson Lemma. We prove that $1+ac$ has Hirano inverse if and only if $1+ba$ has Hirano inverse. Related multiplicative and additive results are also obtained.

Throughout, all rings are associative with an identity. We use $N(R)$ to denote the set of all nilpotent elements in $R$.
The double commutant of $a\in R$ is defined by $comm^2(a)=\{x\in R~|~xy=yx~\mbox{if}~ay=ya~\mbox{for}~y\in R\}$. ${\Bbb N}$ stands for the set of all natural numbers.

\section{Drazin inverse}

The aim of this section is to investigate the relations among Drazin inverse, strongly Drazin inverse and Hirano inverse. We have

\begin{thm} Let $R$ be a ring and $a\in R$. If $a$ has Hirano inverse, then $a$ has Drazin inverse.\end{thm}
\begin{proof} Suppose that $a$ has Hirano inverse $b$. Then $a^2-ab\in N(R)$, $ab=ba$ and $bab=b$.
Hence, $a^2-a^2b^2=a^2-a(bab)=a^2-ab\in N(R)$, and so $$a^2(1-a^2b^2)=(a^2-a^2b^2)(1-a^2b^2)\in N(R).$$ It follows that $a(a-a^2b)=a^2(1-a^2b^2)\in N(R)$; hence, $(a-a^2b)^2=a(a-a^2b)(1-ab)\in N(R)$. Thus, $a-a^2b\in N(R)$. Moreover, we see that
$ab=ba$ and $bab=b$. Therefore, $a$ has Drazin inverse $b$, as desired.\end{proof}

\begin{cor} Let $R$ be a ring and $a\in R$. Then $a$ has at most one Hirano inverse in $R$, and if the Hirano inverse
of $a$ exists, it is exactly the Drazin inverse of $a$, i.e., $a^H=a^D$.\end{cor}
\begin{proof} Let $a$ has Hirano inverse $x$. In view of Theorem 2.1, $a$ has Drazin inverse. Furthermore, we see that $a^D=a^H$. By virtue of~\cite[Theorem 2.4]{W}, the Drazin inverse of $a$ is unique. Thus, $a$ has at most one Hirano inverse in $R$, as desired.\end{proof}

\begin{exam} Every $2\times 2$ matrix over ${\Bbb Z}_2$ has Drazin inverse, but $a=\left(
\begin{array}{cc}
0&1\\
1&1
\end{array}
\right)\in M_2({\Bbb Z}_2)$ has no Hirano inverse.\end{exam}
\begin{proof} Since $M_2({\Bbb Z}_2)$ is a finite ring, every element in $M_2({\Bbb Z}_2)$ has Drazin inverse. If $a=\left(
\begin{array}{cc}
0&1\\
1&1
\end{array}
\right)\in M_2({\Bbb Z}_2)$ has Hirano inverse, it follows by Corollary 2.2 that $a^H=a^D=\left(
\begin{array}{cc}
-1&1\\
1&0
\end{array}
\right)$. But $a^2-aa^H=\left(
\begin{array}{cc}
0&1\\
1&1
\end{array}
\right)$ is not nilpotent. This gives a contradiction, and we are through.\end{proof}

\begin{thm} Let $R$ be a ring and $a\in R$. Then $a$ has Hirano inverse if and only if $a^2$ has strongly Drazin inverse. In this case, $(a^2)^{sD}=(a^H)^2$ and $a^H=a(a^2)^{sD}$.\end{thm}
\begin{proof} $\Longrightarrow$ Since $a$ has Hirano inverse, we have $b\in R$ such that $a^2-ab\in N(R), ab=ba$ and $b=bab$. Thus, $a^2-a^2b^2=a^2-a(bab)=a^2-ab\in N(R), a^2b^2=b^2a^2$ and $b^2a^2b^2=b^2$. Therefore $a^2\in R$ has strongly Drazin inverse. Moreover, $(a^2)^{sD}=(a^H)^2$.

$\Longleftarrow$ Since  $a^2$ has strongly Drazin inverse, we have $x\in R$ such that $a^2-a^2x\in N(R), ax=xa$ and $x=xa^2x$. Set $b=ax$. Then
$a^2-ab=a^2-a^2x\in N(R), ab=ba$ and $bab=(ax)a(ax)=a(xa^2x)=ax=b$. Hence, $a$ has Hirano inverse $b$. Furthermore, $a^H=a(a^2)^{sD}$.\end{proof}

\begin{exam} Every element in ${\Bbb Z}_3$ has Hirano inverse, but $2\in {\Bbb Z}_3$ has no strongly Drazin inverse.
\end{exam}
\begin{proof} One directly checks that every element in ${\Bbb Z}_3$ has Hirano inverse.
If $2\in {\Bbb Z}_3$ has strongly Drazin inverse, then $2=2(2^D)$ and $2^D=2(2^D)^2$, whence $-1=-2^D$, and so $2^D=1$. This gives a contradiction.
This shows that $2\in {\Bbb Z}_3$ has no strongly Drazin inverse.\end{proof}

\begin{exam} Let $a=
\left(
\begin{array}{ccc}
-2&3&2\\
-2&3&2\\
1&-1&-1
\end{array}
\right)\in M_3({\Bbb Z})$. Then $a$ has Hirano inverse, but it has no strongly Drazin inverse.\end{exam}
\begin{proof} Clearly, $a=a^3$, and so $a^2$ has strongly Drazin inverse $a^2$, thus $a$ has Hirano inverse, by Theorem 2.1. On the other hand, the
characteristic polynomial $\chi(a-a^2)=t^2(t+2)\not\equiv t^3 (mod~N(R))$. In light of ~\cite[Proposition 4.2]{D}, $a-a^2$ is not nilpotent.
Therefore $A$ does not have strongly Drazin inverse by~\cite[Lemma 2.2]{W}.\end{proof}

Thus, we have the relations of various type of inverses in a ring by the following: $$\{\mbox{strongly Drazin inverses} \}\subsetneq \{ \mbox{Hirano inverses} \}\subsetneq \{ \mbox{Drazin inverses} \}.$$

\begin{exam} Let ${\Bbb C}$ be the field of complex numbers.

$(1)$ $2\in {\Bbb C}$ is invertible, but $2$ has no Hirano inverse in ${\Bbb C}$.

$(2)$ $A\in M_n({\Bbb C})$ has Hirano inverse if and only if there exists a decomposition ${\Bbb C}^n\oplus {\Bbb C}^n=U\oplus V$ such that $U,V$ are $A$-invariant and
$A^2\mid_{U}, (I_n-A^2)\mid_{V}$ are both nilpotent if and only if eigenvalues of $A^2$ are only $0$ or $1$.
\end{exam}
\begin{proof} $(1)$ Straightforward.

$(2)$ In view of Theorem 2.4, $A\in M_n({\Bbb C})$ has Hirano inverse if and only if $A^2\in M_n({\Bbb C})$ has strongly Drazin inverse if and only if $A^2$ is the sum of an idempotent and an nilpotent matrices that commute. Therefore we are done by~\cite[Lemma 2.2]{W}.\end{proof}

\section{Characterizations}

The purpose of this section is to characterize Hirano inverses by means of tripotents and nilpotents in a ring. The following result is crucial.

\begin{thm} Let $R$ be a ring and $a\in R$. Then $a$ has Hirano inverse if and only if $a-a^3\in N(R)$.
\end{thm}
\begin{proof} $\Longrightarrow$ Since $a$ has Hirano inverse $b$, we have $$w:=a^2-ab\in N(R), ab=ba~\mbox{and}~b=bab.$$ Hence, $a^2=ab+w$ and $(ab)w=w(ab)$. Thus, $a^2-a^4=(ab+w)-(ab+w)^2=(w-2ab-w)w\in N(R)$. Hence, $a(a-a^3)\in N(R)$, and so $(a-a^3)^2=a(a-a^3)(1-a^2)\in N(R)$. Accordingly, $a-a^3\in N(R)$.

$\Longleftarrow$ Suppose that $a-a^3\in N(R)$. Then $a^2-a^4=a(a-a^3)\in N(R)$. In view of~\cite[Lemma 2.1]{CS}, there exists an idempotent $e\in {\Bbb Z}[a^2]$ such that $w:=a^2-e\in N(R)$ and $ae=ea$.
Set $c=(1+w)^{-1}e$. Then $ac=ca$ and $$\begin{array}{lll}
a^2c&=&a^2(1+w)^{-1}e\\
&=&(1+w)^{-1}e(a^2+1-e)\\
&=&(1+w)^{-1}e(1+w)\\
&=&e.
\end{array}$$ Hence, $a^2-a^2c=a^2-e=w\in N(R)$. Moreover,
$ca^2c=ce=c$. Set $b=ac$. Then $a^2-ab\in N(R), ab=ba$ and $bab=(ac)a(ac)=(ca^2c)a=ca=b$.
Therefore $a$ has Hirano inverse $b$, as required.\end{proof}

\begin{cor} Let $R$ be a ring and $a\in R$. If $a$ has Hirano inverse, then $a^H$ has Hirano inverse, and that $(a^H)^H=a^2a^H$.\end{cor}
\begin{proof} Write $b=a^H$. Then $w:=a^2-ab\in N(R), ab=ba$ and $b=bab$. In view of Theorem 3.1, $a-a^3\in N(R)$.
One easily checks that $$\begin{array}{lll}
b-b^3&=&b^2a-b^2(bab)\\
&=&b^2a-b^3(a^2-w)=b^2a-b(b^2a)a-b^3w\\
&=&b^2a-b^2(ba)a-b^3w\\
&=&b^2a-b(b^2a)a-b^3w\\
&=&b^2a-b^2(a^2-w)a-b^3w=b^2(a-a^3)+(b-b^3)w\\
&\in & N(R).
\end{array}$$ By using Theorem 3.1 again, $b$ has Hirano inverse. In light of Corollary 2.2, $b=a^D$ and $b^H=b^D$.
Therefore $(a^H)^H=(a^D)^D=a^2a^D=a^2a^H$, in terms of~\cite[Theorem 5.4]{J}. This completes the proof.\end{proof}

We are now ready to prove the following.

\begin{thm} Let $R$ be a ring with $\frac{1}{2}\in R$, and let $a\in R$. Then the following are equivalent:\end{thm}
\begin{enumerate}
\item [(1)] {\it $a\in R$ has Hirano inverse.}
\vspace{-.5mm}
\item [(2)] {\it There exists $p^3=p\in comm^2(a)$ such that $a-p\in N(R)$.}
\vspace{-.5mm}
\item [(3)] {\it There exist two idempotents $e,f\in comm^2(a)$ such that $a+e-f\in N(R)$.}
\end{enumerate}
\begin{proof} $(1)\Rightarrow (3)$ In view of Theorem 3.1, $a-a^3\in N(R)$. Set $b=\frac{a^3+a}{2}$ and $c=\frac{a^3-a}{2}$. Then $$\begin{array}{c}
b^2-b=\frac{1}{4}(a^4+2a^3-a^2-2a)=\frac{1}{4}(a+2)(a^3-a),\\
c^2-c=\frac{1}{4}(a^4-2a^3-a^2+2a)=\frac{1}{4}(a-2)(a^3-a).
\end{array}$$ Thus, $b-b^2, c-c^2\in N(R)$. In view of~\cite[Lemma 2.1]{CS}, there exist two idempotents $e\in {\Bbb Z}[b]$ and $f\in {\Bbb Z}[c]$ such that $b-e,c-f\in N(R)$.
One easily checks that $a=b-c$ and $b,c\in {\Bbb Z}[a]$. Therefore $a-f+e=(b-f)-(c-e)\in N(R)$. We see that $e,f\in {\Bbb Z}[a]\subseteq comm^2(a)$, as desired.

$(3)\Rightarrow (2)$ Set $p=f-e$. Then $a-p\in N(R)$ and $p^3=(e-f)^3=e-f=p$, as $ef=fe$.

$(2)\Rightarrow (1)$ By hypothesis, there exists $p^3=p\in comm^2(a)$ such that $w:=a-p\in N(R)$. Thus, $a-a^3=(p+w)-(p+w)^3=w(1-3pw-3p^2-w^2)\in N(R)$. THis completes the proof, by Theorem 3.1.\end{proof}

\begin{cor} Let $R$ be a ring with $\frac{1}{2}\in R$. Then $a\in R$ has Hirano inverse if and only if there exist commuting $b,c\in R$ such that $a=b-c$, where $b,c$ have strongly Drazin inverses.\end{cor}
\begin{proof} $\Longrightarrow$ In light of Theorem 3.3, there exist two idempotents $e,f\in comm^2(a)$ such that $w:=a+e-f\in N(R)$. Hence,
$a=f-(e+w)$ with $w\in {\Bbb Z}[a]$. Clearly, $e\in R$ has strongly Drazin inverse.
Furthermore, we see that $(e+w)-(e+w)e\in N(R), e(e+w)=(e+w)e$ and $e=e(e+w)e$. Thus, $e+w\in R$ has strongly Drazin inverse.
Set $b=f$ and $c=e+w$. Then $a=b-c$ and $bc=ca$, as desired.

$\Longleftarrow$ Suppose that $a=b-c$, where $b,c$ have strongly Drazin inverses and $bc=cb$. Then we have $x\in R$ such that $w:=b-bx\in N(R), bx=xb$ and $x=xbx$.
Thus, $b-b^2=bx-b^2+w=b^2x^2-b^2+w=(w+b)^2-b^2+w=(w+2b+1)w\in N(R)$, s $w(w+2b+1)=(w+2b+1)w$. Likewise, $c-c^2\in N(R)$. In view of~\cite[Lemma 2.1]{CS}, there exist $e\in {\Bbb Z}[b]$ and $f\in {\Bbb Z}[c]$ such that $u=b-e,v=c-f\in N(R)$. Thus,
$a=(e+u)-(f+v)=e-f+(u-v)$. As $bc=cb$, we see that $e,f,u,v$ commute one another. It is easy to verify that $(e-f)^3=e-f$, and so $a-a^3\in N(R)$. In light of Theorem 3.1, $a\in R$ has Hirano inverse. This completes the proof.\end{proof}

As $2$ is invertible in any Banach algebra, we now derive

\begin{cor} Let $A$ be a Banach algebra and let $a\in A$. Then the following are equivalent:
\end{cor}\begin{enumerate}
\item [(1)] {\it $a$ has Hirano inverse.}
\vspace{-.5mm}
\item [(2)] {\it There exists $p^3=p\in comm^2(a)$ such that $a-p\in N(A)$.}
\vspace{-.5mm}
\item [(3)] {\it There exist two idempotents $e,f\in comm^2(a)$ such that $a+e-f\in N(A)$.}
\vspace{-.5mm}
\item [(4)] {\it There exist commuting $b,c\in R$ such that $a=b-c$, where $b,c$ have strongly Drazin inverses.}
\end{enumerate}

Recall that a ring $R$ is strongly 2-nil-clean if every element in $R$ is the sum of a tripotent and a nilpotent that commute (see~\cite{CS}).
We characterize strongly 2-nil-clean rings by using Hirano inverses.

\begin{prop} A ring $R$ is strongly 2-nil-clean if and only if every element in $R$ has Hirano inverse.\end{prop}
\begin{proof} $\Longrightarrow $ Let $a\in R$. In view of~\cite[Theorem 2.3]{CS}, $a-a^3\in N(R)$. Thus, $a$ has Hirano inverse by Theorem 3.1.

$\Longleftarrow $ Let $a\in R$. In light of Theorem 3.1, $a-a^3\in N(R)$. Therefore we complete the proof, by~\cite[Theorem 2.3]{CS}.\end{proof}

\begin{cor} A Banach algebra $A$ is strongly 2-nil-clean if and only if every element in $A$ is the difference of two commuting elements
that have strongly Drazin inverses.\end{cor}
\begin{proof} Combining Proposition 3.6 and Corollary 3.5, we obtain the result.\end{proof}

\section{Clines formula}

Let $a,b\in R$. Then $ab$ has Drazin inverse if and only if $ba$ has Drazin inverse. In this case,
$(ba)^D = b((ab)^D)^2a$. This was known as Clines formula for Drazin inverses.
It plays an important role in
revealing the relationship between the Drazin inverse of a sum of two elements and the
Drazin inverse of a block triangular matrix. In ~\cite{ZZ} extended this result and proved that if $aba=aca$ then $1+ac$ has Drazin inverse if and only if $ba$ has Drazin inverse.
The goal of this section is to generalize Cline¡¯s formula from Drazin inverses to Hirano inverses. We derive

\begin{thm} Let $R$ be a ring, and let $a,b,c\in R$. If $aba=aca$, Then $ac$ has Hirano inverse if and only if $ba$ has Hirano inverse. In this case, $(ba)^{H}=b((ac)^{H})^2a$.
\end{thm}
\begin{proof} $\Longrightarrow$ By induction on $n$ we shall prove 
$$(ba-(ba)^3)^{n+1}=b(ac-(ac)^3)^n(a-ababa).$$
 For $n=1$, $$\begin{array}{l}
((ba-(ba)^3))^2\\
=(ba-bababa)(ba-bababa)\\
=b(a-ababa)ba(1-baba)\\
=b(aba-abababa)(1-baba)\\
=b(aca-acacaca)(1-baba)\\
=b(ac-acacac)(a-ababa)\\
=b(ac-(ac)^3)(a-ababa).
\end{array}$$
Suppose that the assertion is true for any $n\leq k (k\geq 2)$. Then,

$$\begin{array}{l}
(ba-(ba)^3)^{k+1}\\
=((ba-(ba)^3))^{k} (ba-(ba)^3)\\
=b(ac-(ac)^3)^{k-1}(a-ababa)(ba-bababa)\\
=b(ac-(ac)^3)^{k-1}(aba-abababa)(1-baba)\\
=b(ac-(ac)^3)^{k-1}(aca-acacaca)(1-baba)\\
=b(ac-(ac)^3)^{k}(a-ababa).
\end{array}$$
  As $ac$ has Hirano inverse, then by Theorem 3.1, $(ac-(ac)^3)^n=0$ for some $n\in {\Bbb N}$, which implies that $ba-(ba)^3\in N(R))$ by the preceding discussion. In light of Theorem 3.1, $ba$ has Hirano inverse.
  
Now let $e=(ac)^H$, by Corollary 2.2, $e$ is Drazin inverse of $ac$. As we proved, $ba$ has Hirano inverse, and its Hirano inverse is its drazin inverse, then by \cite[Theorem 2.7]{ZZ}, $(ba)^D =(b(ac)^2)^Da$ and so $(ba)^H=b(b(ac)^2)^H)a$.

$\Longleftarrow$ This is symmetric.\end{proof}

\begin{cor} Let $R$ be a ring and $a,b\in R$. If $ab$ has Hirano inverse, then so has $ba$, and $(ba)^{H} = b((ab)^{H})^2a.$\end{cor}
\begin{proof} It follows directly from Theorem 4.1, as $aba=aba$.\end{proof}

\begin{thm} Let $R$ be a ring, and let $a,b\in R$. If $(ab)^k$ has Hirano inverse, then so is $(ba)^k$.
\end{thm}
\begin{proof} $\Longrightarrow$ Let $k=1$, then by Corollary 4.2, $ab$ has Hirano inverse if and only if $ba$ has Hirano inverse. Now let $k\geq 2$, as $(ab)^k=a(b(ab)^{k-1})$, then by Corollary 4.2, $(b(ab)^{k-1})a$ has Hirano inverse, so $(ba)^k$ has Hirano inverse.
\end{proof}

\begin{thm} Let $R$ be a ring, and let $a,b\in R$. If $a,b$ have Hirano inverses and $ab=ba$. Then $ab$ has Hirano inverse and $$(ab)^H=a^Hb^H.$$
\end{thm}
\begin{proof} Since $a,b$ have Hirano inverses, by Theorem 3.1, we have $a-a^3,b-b^3\in N(R)$. Note that
$$\begin{array}{lll}
ab-(ab)^3&=&(a-a^3)(b-b^3)+ab^3+a^3b-a^3b^3-a^3b^3\\
&=&(a-a^3)(b-b^3)+a^3(b-b^3)+b^3(a-a^3)\\
&\in& N(R),
\end{array}$$
as $ab=ba$. By using Theorem 3.1 again, $ab$ has Hirano inverse. By Theorem 2.1, $ab$ has Drazin inverse and by \cite[Lemma 2]{ZC}, $(ab)^D=b^Da^D$. According to Corollary 2.2, $(ab)^H=(ab)^D=b^Da^D=b^Ha^H$, as asserted.\end{proof}

\begin{cor} Let $R$ be a ring. If $a\in R$ has Hirano inverse, then $a^n$ has Hirano inverse and $(a^n)^H=(a^H)^n$ for all $n\in {\Bbb N}$.
 \end{cor}
\begin{proof} By Theorem 4.4 and induction, we easily obtain the result.\end{proof}

We note that the converse of the preceding corollary is not true, as the following shows.

\begin{exam} Let ${\Bbb Z}_5={\Bbb Z}/5{\Bbb Z}$ be the ring of integers modulo $5$. Then $-2\in {\Bbb Z}_5$ has no Hirano inverse, as $(-2)-(-2)^3=1\in {\Bbb Z}_5$ is not nilpotent, but $(-2)^2=-1\in {\Bbb Z}$ has Hirano inverse, as $(-1)-(-1)^3\in {\Bbb Z}_5$ is nilpotent.
\end{exam}

\section{Additive properties}

Let $a,b\in R$. Jacobson Lemma asserted that $1+ab\in R$ has Drazin inverse if and only if $1+ba\in R$ has Drazin inverse.
More recently, Yan and Fang studied the link between basic operator properties of $1+ac$ and $1+ba$ when $aba=aca$ (see~\cite{YF}). Though we do not know 
 whether the Drazin invertibility of $1+ac$ and $1+ba$ coincide under $aba=aca$, we could generalize Jacobson Lemma for Hirano inverses. 

\begin{thm} Let $R$ be a ring, and let $a,b,c\in R$. If $aba=aca$, Then $1+ac$ has Hirano inverse if and only if $1+ba$ has Hirano inverse.
\end{thm}
\begin{proof} $\Longrightarrow$   First by induction on $n$ we prove that
$$((1+ba)-(1+ba)^3)^{n+1}=b((1+ac)-(1+ac)^3)^n(-2a-3aba-ababa).$$
 For $n=1$,
 $${\footnotesize\begin{array}{l}
 ((1+ba)-(1+ba)^3)^2\\
 =((1+ba)-(1+ba)^3)((1+ba)-(1+ba)^3)\\
 =(-2ba-3baba-bababa)(1+ba-1-3ba-3baba-bababa)\\
 =(-2ba-3baba-bababa)ba(-2-3ba-baba)\\
 =(-2baba-3bababa-babababa)(-2-3ba-baba)\\
 =b(-2aba-3ababa-abababa)(-2-3ba-baba)\\
 =b(-2ac-3acac-acacac)(-2a-3aba-ababa)\\
 =b((1+ac)-(1+ac)^3)(-2a-3aba-ababa).
 \end{array}}$$
 Let $k\geq 2$ and the assertion works for any $n\leq k$. Then
 $${\footnotesize\begin{array}{l}
 (( 1+ba)-(1+ba)^3)^{k+1}\\
 =((1+ba)-( 1+ba)^3)^{k}((1+ba)-(1+ba)^3)\\
 = b((1+ac)-(1+ac)^3)^{k-1}(-2a-3ba-baba)((1+ba)-(1+ba)^3) \\
 = b((1+ac)-(1+ac)^3)^{k-1}(-2a-3ba-baba)(ba-3ba-3baba-(ba)^3)\\
  = b((1+ac)-(1+ac)^3)^{k-1}(-2aba-3baba-bababa)(-2-3ba-baba)\\
   = b((1+ac)-(1+ac)^3)^{k}(-2a-3aba-ababa).
   \end{array}}$$
  As $1+ac$ has Hirano inverse, then by Theorem 3.1, $(1+ac)-(1+ac)^3\in N(R)$. By the above equality, $(1+ba)-(1+ba)^3\in N(R)$. In light of Theorem 3.1, $(1+ba)$ has Hirano inverse. 

$\Longleftarrow$ This is symmetric.\end{proof}

\begin{cor} Let $R$ be a ring, and let $a,b\in R$. Then $1+ab$ has Hirano inverse if and only if $1+ba$ has Hirano inverse.
\end{cor}
\begin{proof} The proof follows directly from the equality $aba=aba$ and Theorem 5.1.\end{proof}

As is well known, $a\in R$ has strongly Drazin inverse if and only if so does $1-a\in R$ (see~\cite[Lemma 3.3]{W}). But Hirano inveses in a ring are not the case.

\begin{exam} Let $a=\left(
\begin{array}{cc}
0&1\\
-1&0
\end{array}
\right)\in M_2({\Bbb Z}_3)$. Then $a$ has Hirano inverse, but $1-a$ does not have.\end{exam}
\begin{proof} It is easy to verify that $a-a^3=\left(
\begin{array}{cc}
1&1\\
-1&2
\end{array}
\right)\in N(M_2({\Bbb Z}_4))$, but $(1-a)-(1-a)^3=\left(
\begin{array}{cc}
0&1\\
-1&0
\end{array}
\right)\not\in N(M_2({\Bbb Z}_4))$. Therefore we are done by Theorem 3.1.\end{proof}

\begin{prop} Let $R$ be a ring and $a,b\in R$. If $a,b$ have Hirano inverses and $ab=ba=0$, then $a+b$ has Hirano inverse and $$(a+b)^{H}=a^H+b^H.$$\end{prop}
\begin{proof} Clearly, $ab^H=ab(b^H)^2=0$ and $ba^H=ba(a^H)^2=0.$ Likewise, $a^Hb=b^Ha=0$. Thus, we see that
$a,a^H,b,b^H$ commute. Thus,
$(a^2-aa^H)(b^2-bb^H)=(b^2-bb^H)(a^2-aa^H).$
Also we have $$(a+b)(a^H+b^H)^2=(aa^H+bb^H)(a^H+b^H)=a^H+b^H$$
and $$(a+b)^2-(a+b)(a^H+b^H)=(a^2-aa^H)+(b^2-bb^H)\in N(R).$$
Therefore $(a+b)^H=a^H+b^H$, as asserted,.\end{proof}

It is convenient at this stage to include the following. 

\begin{thm} Let $R$ be a ring, and let $a,b\in R$. If $ab$ has strongly inverse and $a^2=b^2=0$, then $a+b$ has Hirano inverse and $(a+b)^{H}=a(ba)^{H}+b(ab)^{H}$.
\end{thm}
\begin{proof} Suppose that $ab$ has strongly inverse and $a^2=b^2=0$. Then $ba$ has strongly inverse by~\cite[Theorem 3.1]{W}. It is obvious that $ab$ and $ba$ have Drazin inverses. Let $x=a(ba)^{D}+b(ab)(ab)^{D}$. Then we easily check that $(a+b)x=x(a+b)$ and
$(a+b)x^2=x$. We easily check that $$\begin{array}{lll}
(a+b)^2-(a+b)x&=&ab+ba-(ab)(ab)^D-(ba)(ba)^D\\
&=&(ab-(ab)(ab)^D)+(ba-(ba)(ba)^D)\\
&\in &N(R),
\end{array}$$ as $(ab-(ab)(ab)^D)(ba-(ba)(ba)^D)=(ba-(ba)(ba)^D)(ab-(ab)(ab)^D)$ $=0.$ Therefore $a+b\in R$ has Hirano inverse, and that $(a+b)^H=a(ba)^{D}+b(ab)^{D}=a(ba)^{H}+b(ab)^{H}$.\end{proof}

\begin{exam} Let $a=\left(
\begin{array}{cc}
0&1\\
0&0
\end{array}
\right), b=\left(
\begin{array}{cc}
0&0\\
2&0
\end{array}
\right)\in M_2({\Bbb Z}_3)$. Then $ab=\left(
\begin{array}{cc}
2&0\\
0&0
\end{array}
\right)\in M_2({\Bbb Z})$ has Hirano inverse, $a^2=b^2=0$. But $a+b=\left(
\begin{array}{cc}
0&1\\
2&0
\end{array}
\right)\in M_2({\Bbb Z}_3)$ has no Hirano inverse.\end{exam}

\begin{exam} Let $a=\left(
\begin{array}{cc}
0&1\\
0&0
\end{array}
\right), b=\left(
\begin{array}{cc}
0&0\\
1&0
\end{array}
\right)\in M_2({\Bbb Z})$. Then $ab=\left(
\begin{array}{cc}
1&0\\
0&0
\end{array}
\right)\in M_2({\Bbb Z})$ has strongly Drazin inverse, $a^2=b^2=0$. But $a+b=\left(
\begin{array}{cc}
0&1\\
1&0
\end{array}
\right)$ has no strongly Drazin inverse. In this case, $(a+b)^H=\left(
\begin{array}{cc}
0&1\\
1&0
\end{array}
\right).$\end{exam}

\end{document}